\documentclass[11pt]{amsart}
\usepackage{ae,amsfonts,euscript,enumerate}
\usepackage{amsmath,euscript}
\usepackage{amssymb}
\usepackage{amsthm}
\usepackage{enumerate}
\usepackage{amsfonts}
\usepackage{comment}
\usepackage{colortbl}
\usepackage{amssymb,amsmath,array}

\newtheorem{thm}{Theorem}[section]
\newtheorem{lem}[thm]{Lemma}

\newtheorem{prop}[thm]{Proposition}
\newtheorem{cor}[thm]{Corollary}

\theoremstyle{definition}

\newtheorem{rem}[thm]{Remark}

\newcommand{\dra}{\dashrightarrow}

\def\co{{\mathcal O}}

\def\oqmm13{\co_q(M_{1,3})}
\def\oqm23{\co_q(M_{2,3})}

\usepackage{amssymb}
\usepackage{amsmath}
\usepackage[latin1]{inputenc}

\usepackage{amsmath,euscript}
\usepackage{amssymb}
\usepackage{amsthm}
\usepackage{enumerate}
\usepackage{amsfonts}
\usepackage{comment}
\usepackage{colortbl}
\usepackage{a4wide}
\usepackage{amssymb,amsmath,array}

\begin{document}

\title[]{Free algebras and free groups in Ore extensions and free group algebras in division rings}

\author{Jason P.~Bell}
\thanks{The first author was supported by NSERC Grant 326532-2011.  The second author was supported by Grant CNPq 301.320/2011-0, and FAPESP-Brazil, Proj. Tematico 2009/52665-0.}


\address{Jason P. Bell\\
Department of Pure Mathematics\\
University of Waterloo\\
Waterloo, ON, N2L 3G1\\
Canada}

\email{jpbell@uwaterloo.ca}

\author{Jairo Z. Gon\c calves}

\address{Jairo Z. Gon\c calves\\
Department of Mathematics\\
University of S\~ao Paulo\\
S\~ao Paulo, 05508-090 \\
Brazil}
\email{jz.goncalves@usp.br}

\bibliographystyle{plain}

\begin{abstract} Let $K$ be a field of characteristic zero, let $\sigma$ be an automorphism of $K$ and let $\delta$ be a $\sigma$-derivation of $K$.  We show that the division ring $D=K(x;\sigma,\delta)$ either has the property that every finitely generated subring satisfies a polynomial identity or $D$ contains a free algebra on two generators over its center.  In the case when $K$ is finitely generated over $k$ we then see that for $\sigma$ a $k$-algebra automorphism of $K$ and $\delta$ a $k$-linear derivation of $K$, $K(x;\sigma)$ having a free subalgebra on two generators is equivalent to $\sigma$ having infinite order, and $K(x;\delta)$ having a free subalgebra is equivalent to $\delta$ being nonzero.  As an application, we show that if $D$ is a division ring with center $k$  of characteristic zero and $D^*$ contains a solvable subgroup that is not locally abelian-by-finite, then $D$ contains  a free $k$-algebra on two generators. Moreover, if we assume that $k$ is uncountable, without any restrictions on the characteristic  of $k$, then $D$ contains the $k$-group algebra of the free group of rank two.
\end{abstract}
\maketitle
\section{Introduction}  
There has been much study on free subobjects in division rings, both in the case of free subalgebras \cite{BR1, BR2, FGM05, FGS, Licht, Lor, ML, ML2, ML3, SG} and in the case of free subgroups and free subsemigroups of their multiplicative groups \cite{Chiba96, GL15, GP14, Lichtman78, ML15, RV}.  (See also the survey \cite{GS12}.) In this paper, we study the question of when division algebras formed from Ore extensions of fields contain free subalgebra in two generators over their centers.   We recall that given a ring $R$ with an automorphism $\sigma$ of $R$ and a $\sigma$-derivation, which is a map $\delta: R\to R$ satisfying $\delta(ab)=\sigma(a)\delta(b)+\delta(a)b$ for all $a,b\in R$, we
can form an \emph{Ore extension}, $R[x;\sigma,\delta]$, of $R$, which as a set is just the polynomial ring $R[x]$, but which is endowed with multiplication given by $xr=\sigma(r)x+\delta(r)$ for $r\in R$.

Makar-Limanov gave the first result in this direction by showing that if $A=k\{x,y\}/(xy-yx-1)$ is the Weyl algebra over a field $k$ of characteristic $0$, then its quotient division ring contains a free $k$-algebra on two generators \cite{ML}.   Since the first Weyl algebra can be realized as a skew extension $k[x][y;\delta]$ where $\delta$ is a derivation of $k[x]$ with $\delta(x)=1$, we see that this fits into the framework of Ore extensions.  

 The belief is that this phenomenon of division rings containing free subalgebras holds very generally, and should only fail to hold when there is an obvious obstruction, such as being commutative or being algebraic over its center.  We call the \emph{free subalgebra conjecture} the statement that a division algebra $D$ must contain a free algebra on two generators over its center unless $D$ is locally PI, that is, all of its affine subalgebras are polynomial identity algebras.  This conjecture (in some form) was formulated independently by both Makar-Limanov and Stafford.    We refer the reader to \cite{BR1} and \cite{GS12} for a more detailed discussion of the conjecture, and past work on the subject.  
 
 Unfortunately, the free subalgebra conjecture is considered to be a very hard problem.  Indeed, it is closely related to the notoriously difficult analogue of the Kurosh problem for division rings, which asks whether division rings that are finitely generated and algebraic over their centers, but that are not finite-dimensional over them, can exist.  This Kurosh analogue has remained open for some time, and a proof of the free subalgebra conjecture would give a solution to the Kurosh problem as well.  In the case of free subgroups of the multiplicative group of a division ring more is known.  In fact, Chiba \cite{Chiba96} has shown that if the base field is uncountable then a division ring that is not commutative has a free subgroup inside its multiplicative group.  For free algebras, however, the results appear to be harder to obtain.  

The first author and Rogalski showed \cite{BR1} that the free subalgebra conjecture holds for the quotient division rings of iterated Ore extensions of PI rings, if the base field $k$ is uncountable, or if the base field is countable and the associated field automorphism of $K$ is ``geometric'' in the sense of being induced by a regular automorphism of a variety whose field of rational functions is equal to $K$.  This work left an obvious gap: namely, answering the free subalgebra conjecture for division rings of the form $K(x;\sigma)$ where $\sigma$ is an automorphism of $K$ that is not geometric and such that the fixed field of $\sigma$ is countable.  This is a much more difficult case to work with because it is difficult to apply the combinatorial counting methods available when working over an uncountable base field and, likewise, it is difficult to apply geometric techniques in this setting when the automorphism is not geometric.

Our main theorem (see Corollary \ref{thm: main1}) shows how to deal with these issues when the base field is of characteristic zero.  In particular, we show that given a field $K$ of characteristic zero, then the quotient division ring of an Ore extension $K[x;\sigma,\delta]$, either contains a free algebra on two generators over its center or it is locally PI.  
\begin{thm}
\label{thm: main}
Let $K$ be a field of characteristic zero, let $\sigma$ be an automorphism of $K$ and let $\delta$ be a $\sigma$-derivation of $K$.  Then either $D:=K(x;\sigma,\delta)$ is locally PI or $D$ contains a free algebra on two generators over its center.
\end{thm}
In fact, every division ring of the form $K(x;\sigma,\delta)$ is isomorphic to a division ring of the form $K(x;\tau)$ for some automorphism $\tau$ of $K$, or to a division ring of the form $K(x;\mu)$ where $\mu$ is a derivation of $K$ (see the proof of Corollary \ref{thm: main1} for details).  Then, in the characteristic zero setting, one has that $K(x;\mu)$ is locally PI if and only if $\mu$ is zero, and in the special case when $K$ is finitely generated as an extension of $k$, $K(x;\tau)$ is locally PI if and only if it is finite-dimensional over its center and this occurs if and only if $\tau$ has finite order (see Conjecture \ref{thm: main1} for details).  Thus we have the following corollary of Theorem \ref{thm: main}.
\begin{cor}
Let $k$ be a field of characteristic zero, let $K$ be a finitely generated extension of $k$.  Then we have the following:
\begin{enumerate}
\item[(1)] if $\sigma$ is a $k$-algebra automorphism of $K$ then $K(x;\sigma)$ contains a free algebra over $k$ if and only if $\sigma$ has infinite order;
\item[(2)] if $\delta$ is a $k$-linear derivation of $K$ then $K(x;\delta)$ contains a free algebra over $k$ if and only if $\delta$ is nonzero.
\end{enumerate}
\end{cor}
We note that (2) is already well-known, but (1) is new and appears to require algebro-geometric machinery as well as deep model theoretic work of Hrushovski \cite{H}, in order to obtain this result.  

As an application of Theorem \ref{thm: main}, we can prove a result about the existence of free subalgebras in division rings whose multiplicative groups contain solvable subgroups that have some finitely generated subgroup that is not abelian-by-finite (see Theorem \ref{thm: main2}).   
\begin{thm}
\label{thm: main0}
Let $D$ be a division ring with center $k$, and suppose that $G\le D^*$ is solvable but not locally abelian-by-finite.  Then the following statements hold:

\begin{enumerate} 
 \item if $k$ has characteristic $0$   then $D$ contains a free  $k$-algebra on two generators , and
 \item if $k$ is uncountable then $D$ contains the $k$-group algebra of the free group of rank two.
\end{enumerate}
\end{thm}
Previously this had only been shown when $G$ is polycyclic-by-finite and not abelian-by-finite \cite{L}.  We note that in the polycyclic-by-finite case, being locally abelian-by-finite is equivalent to being abelian-by-finite.  

Our last result is to consider Lichtman's question on the existence of free subgroups of $D^*=D \setminus \{0\}$, the multiplicative groups of a division ring   $D$  of the form $K(t;\sigma,\delta)$.  As is standard, one can reduce this question to the study of division rings of the form $K(t;\sigma)$ and $K(t;\delta)$ where $\sigma$ and $\delta$ are respectively an automorphism of $K$ and a derivation of $K$.  In the case when $\delta$ is a nonzero derivation, then we show that $K(t;\delta)^*$ always contains a free non-cyclic subgroup.  When $K(t;\sigma)$ is not a field and $\sigma$ either has finite order, or has infinite order and its fixed field has infinite transcendence degree over the prime field, we show that $K(t;\sigma)^*$ contains a free non-cyclic subgroup.  Thus for Ore extensions of this form, the only remaining question to answer is whether there are free non-cyclic subgroups of $D^*$, when 
$D=K(t;\sigma)$ and $\sigma$ has infinite order and its fixed field is of finite transcendence degree over the prime field.

The proof of Theorem \ref{thm: main} uses work of Amerik \cite{Amerik}, which is in turn used work of Hrushovski \cite{H}, and the work of the first author and Ghioca and Tucker \cite{BGT}.  In particular, given a field $k$ of characteristic zero and a finitely generated extension $K$ of $k$, a $k$-algebra automorphism $\sigma$ of $K$ corresponds to a birational map $\phi: X\dra X$ of a normal projective variety defined over $k$, whose field of rational functions is $K$.  Amerik shows that if $\sigma$ has infinite order then there is some point $x\in X(\bar{k})$ whose orbit under $\phi$ is defined at every point and such that the orbit is infinite.  Moreover, she shows that if one replaces $\phi$ by a suitable iterate, the orbit can be parametrized by a $p$-adic analytic arc and thus, one can use tools from $p$-adic analysis such as Strassman's theorem to study the intersection of the orbit with closed subsets of $X$.  We make use of this to show that one can apply a criterion of the first author and Rogalski \cite{BR1}, to obtain the existence of free algebras.  

The outline of this paper is as follows.  In \S 2, we describe the work of Amerik and how it will apply.  In \S 3, we use the results of \S 2 to prove Theorem \ref{thm: main}.  In \S 4 we apply Theorem \ref{thm: main} to obtain Theorem \ref{thm: main0}.  Finally, in \S 5 we give some results on free subgroups of multiplicative groups of division rings.


\section{Geometric results}
In this section, we give a description of Amerik's work and how it will apply to our setting.  We begin with a statement of her main result.
\begin{thm} (Amerik \cite{Amerik}) Let $X$ be a projective variety defined over a field $k$ of characteristic zero and let $f:X\dra X$ be a dominant self-map that has infinite order and is also defined over $k$.  Then there exists a point $x\in X(\bar{k})$ such that orbit of $x$ under $f$ is infinite.
\label{thm: a}
\end{thm}
There is an important remark to be made here.  Since $f$ is a rational map, it is not defined at all points in $X$.  In particular, it is highly non-obvious that we should be able to find a point $x$ such that $f$ is defined at $f^n(x)$ for all $n\ge 0$.  Nevertheless, Amerik was able to prove this.   We note that although the result of Amerik is stated for varieties defined over a number field, the argument is easily seen to apply to any field of characteristic zero, and the fact that it does is well-known to experts.  We suspect that the main audience of this work will not be familiar with some of the techniques used, so we give a brief description of her methods and why they apply in this more general setting.  

The way one proves Theorem \ref{thm: a} is by first noting that a projective variety and the map $f$ can all be defined in terms of finitely many polynomials and so we can assume without loss of generality that $k$ is a finitely generated extension of $\mathbb{Q}$ by replacing $k$ by the subfield generated over $\mathbb{Q}$ by the coefficients of these polynomials.  Thus we may reduce to the case when $k$ is a finitely generated extension of $\mathbb{Q}$.  Amerik next shows that by choosing an affine open subset $U$ of $X$ on which $f$ is regular, one can work with a model $\mathcal{U}$ of $U$ over ${\rm Spec}(A)$, where $A$ is a finitely generated $\mathbb{Z}$-algebra whose field of fractions is equal to $k$.  Now since $A$ satisfies the Nullstellensatz \cite[Theorem 4.19]{Eisenbud}, we have a Zariski dense set of maximal ideals of $A$ and the quotient of $A$ by each maximal ideal is a finite field.  By choosing a suitable maximal ideal $Q$ of $A$, Amerik argues that we can reduce mod $Q$ and get a projective variety $\bar{X}$ defined over a finite field $\mathbb{F}_q$ with a rational self-map $\bar{f}$ also defined over $\mathbb{F}_q$.  Amerik now uses an argument of Hrushovski \cite[Corollary 1.2]{H} to show that there is some $\bar{\mathbb{F}}_q$-point of $\bar{X}$ which has a finite orbit under $\bar{f}$ and such that $\bar{f}$ is defined at each point of the orbit.  Moreover, she shows that one can pick a point which has the property that each point in the orbit avoids the ramification locus of $f$.  Now she argues that there is a finite integral extension of $B$ of $A$ such that $x$ lifts to a Zariski dense set of points in $y\in U(B)$ whose reduction mod $Q$ is equal to $x$.  

She now uses a $p$-adic argument of the first author, Ghioca, and Tucker \cite{BGT} to show that there is a Zariski dense set of points $y\in X(\bar{k})$ for which one give a $p$-adic analytic parametrization of the orbit of $y$ under some iterate $g=f^n$ of $f$ with $n\ge 1$.  This is accomplished by embedding $B$ in the $p$-adic integers, $\mathbb{Z}_p$, for a suitable prime $p$ and then using Hensel-type arguments (this strongly uses zero characteristic, but Amerik also implicitly uses characteristic zero in other places in the proof).  These techniques are explained in more detail in \cite{BGT}.  Finally, she shows that at least one such orbit must be infinite if $f$ has infinite order (in fact, she shows that the set of such points is Zariski dense).  Moreover, the $p$-adic parametrization extends to the backwards orbit of $y$ the case that $g$ is birational, and so one has the entire two-way orbit of $y$ under $g$ avoids the indeterminacy locus of $g$ and has a $p$-adic parametrization.  We also note that by these $p$-adic methods (cf. \cite[Theorem 1.3]{BGT}), that after replacing $f$ by a suitable iterate $g$, that if we take the two-way orbit of $y$ under $g$ and let $Y$ be a Zariski closed subset of $X$, then the set of $n$ for which $g^n(y)\in Y$ is either finite or $g^n(y)\in Y$ for every integer $n$.  The reason for this is that the $p$-adic methods of \cite[Theorem 1.3]{BGT} show that there exist a finite set of $p$-adic analytic maps $\theta:\mathbb{Z}_p\to \mathbb{Z}_p$ such that each $\theta(n)=0$ if and only if $g^n(y)\in Y$ and by Strassman's theorem, a $p$-adic analytic map that is not identically zero has at most finitely many zeros on $\mathbb{Z}_p$.   
Thus we have as a corollary, the following result.
\begin{cor}
\label{c: amerik}
Let $X$ be a projective variety defined over a field $k$ of characteristic zero, and let $\phi:X\dra X$ be a birational map that has infinite order and is also defined over $k$.  Then there exists a point $x\in X(\bar{k})$ such that two-way orbit of $x$ under $\phi$ is infinite.  Moreover, there exists some $n\ge 1$ such that if we let $\psi=\phi^n$, then given a Zariski closed subset $Y$ of $X$ we have that the set of integers $n$ for which $\psi^n(x)\in Y$ is either all of $\mathbb{Z}$ or is finite.  \end{cor}
In the next section, we will especially make use of Corollary \ref{c: amerik} to obtain Theorem \ref{thm: main}
\section{Proof of the main result}
In this section we shall prove Theorem \ref{thm: main} (see Corollary \ref{thm: main1}).  To do this, we begin with a technical proposition, which gives a dichotomy which will allow us to obtain our main result.
\begin{prop}
\label{prop: S} Let $k$ be a field of characteristic zero and let $K$ be a finitely generated extension of $k$ of transcendence degree $d\ge 1$, with the property that $k$ is algebraically closed in $K$.  If $\sigma:K\to K$ is a $k$-algebra automorphism then there exists $n\ge 1$ such that, either there exist $f\in K$ and a valuation $\nu:K^*\to \mathbb{Z}$ such that $\nu(f)\neq 0$ and $\nu(\sigma^{nj}(f))=0$ for all but finitely many $j$, or there exists a $k$-subalgebra $S$ of $K$ with the following properties:
\begin{enumerate}
\item[(i)] $S$ is noetherian and integrally closed and has Krull dimension $d$;
\item[(ii)] the field of fractions of $S$ is equal to $K$;
\item[(iii)] $S$ is a localization of a finitely generated $k$-algebra.
\item[(iv)] $\sigma^n$ restricts to a $k$-algebra automorphism of $S$;
\item[(v)] $S$ has a maximal ideal that has infinite orbit under $\sigma^n$. 
\end{enumerate}
\end{prop}
\begin{proof}
Let $X$ be a normal projective variety defined over $k$ with the property that $k(X)=K$.  Then the map $\sigma$ induces a birational map $\phi:X\dra X$.  By Theorem \ref{c: amerik}, we have that there is some point $x\in X(\bar{k})$ such that the two-way orbit of $x$ under $\phi$ is infinite (and in particular $\phi$ is defined at each point in the orbit).  Let $U$ be a affine open subset containing infinitely many points of the orbit of $x$ under $\phi$ and let $Y_1,\ldots , Y_d$ be a finite set of codimension one subvarieties that cover $X\setminus U$ but such that $x$ is not in any of the $Y_i$.  Then
each $Y_i$ gives a rank one discrete valuation $\nu_i$ of $K$. 
 
By Corollary \ref{c: amerik} there is some $n\ge 1$ such that if $\mathcal{O}$ denotes the two-way infinite orbit of $x$ under $\phi^n$, then the set of points in $\mathcal{O}$ that are in $Y_i$ is finite (since $x\not\in Y_i$).  
We take a rational function $f$ that vanishes on $Y_i$ but not at $x$.  Suppose first that there is some point $x'\in \mathcal{O}\cap Y_i$.   Then $Y_i$ induces a valuation $\nu_i$ and we have that $\nu_i(\sigma^{nj}(f))>0$ for $j=0$.  We note that $\nu_i(\sigma^{nj}(f))>0$ then $\sigma^{nj}(f)$ vanishes on $Y_i$ and hence vanishes at $x'$.  In particular, we have that $f$ vanishes at $\sigma^{-nj}(x')$.  Since $x$ is of the form $\sigma^{-nj}(x')$ for some $j$ and $f$ does not vanish at $x$, we see by Corollary \ref{c: amerik} that the set of $j$ for which $\sigma^{-nj}(x')$ is in the zero locus of $f$ is finite.  Thus we see that
$\nu_i(\sigma^{nj}(f))\le 0$ for all but finitely many $j$.  A similar argument using $1/f$ now shows that $\nu_i(\sigma^{nj}(f))\ge 0$ for all but finitely many $j$.  Thus if $\mathcal{O}\cap Y_i$ is non-empty for some $i$ then we see that there exists $f\in K$ and a valuation $\nu:K^*\to \mathbb{Z}$ such that $\nu(f)\neq 0$ and $\nu(\sigma^{nj}(f))=0$ for all but finitely many $j$, as claimed.  

Hence we may assume that $\mathcal{O}$ is completely contained in the affine set $U$.
Now let $S$ denote the set of elements of $K$ that are regular at all points of $\mathcal{O}$.  Then since $\mathcal{O}$ is contained in $U$, we see that $S$ is a localization of the set of functions in $K$ that are regular on $U$, which is a finitely generated $k$-algebra of Krull dimension $d$ whose field of fractions is $K$.  Moreover, $S$ has a localization consisting of the local ring, consisting of elements of $K$ that are regular at the point $x$.  This local ring has Krull dimension $d$.  (We note that the residue field of this local ring will not be equal to $k$ in general, but rather to a finite extension of $k$ since $x$ is a $\bar{k}$-point of $X$.)  Thus $S$ is a localization of a noetherian $k$-algebra of Krull dimension $d$ and has a localization of Krull dimension $d$ and hence $S$ has Krull dimension $d$ and is noetherian and has field of fractions $K$.  Since the two-way orbit of $x$ is invariant under application of $\phi^n$ and $\phi^{-n}$, we see that the restriction of $\sigma^n$ to $S$ gives an automorphism of $S$.  We may replace $S$ by its integral closure if necessary and assume that $S$ is integrally closed.  Since an automorphism of a ring extends to its integral closure we see $\sigma^n$ still gives an automorphism of $S$.  Moreover, since integral closure commutes with localization, we have that $S$ is a localization of the integral closure of the set of functions on $K$ that are regular on $U$, which is a finitely generated $k$-algebra.  This proves (i)--(iv).

The point $x\in X$  corresponds to a maximal ideal $M_x$ of $S$ (by going up and by the correspondence between closed points and maximal ideals in an affine scheme), and by construction this maximal ideal has infinite order under $\sigma^n$, since $\sigma^n$ is the automorphism of $K$ induced by $\phi^n$.  This completes the proof of (v).
\end{proof}
We are now able to prove the automorphism case of Theorem \ref{thm: main}
\begin{thm}
\label{thm: X}
Let $K$ be a field of characteristic zero and let $\sigma$ be an automorphism of $K$ and let $k$ denote the fixed subfield of $\sigma$.  Then if $K$ is a finitely generated extension of $k$ and $\sigma$ has infinite order, then $K(x;\sigma)$ contains a free $k$-algebra on two generators whose valuations with respect to $x$ are nonnegative.  
\end{thm}
\begin{proof}

We prove this by induction on the transcendence degree of $K$ over $k$.  If this transcendence degree is $0$ then $K$ is a finite extension of $k$ and so $\sigma$ has finite order and so the claim is vacuous.  We now assume that the claim holds whenever the transcendence degree is strictly less than $d$, with $d\ge 1$, and we consider the case when $K$ has transcendence degree $d$.  

By Proposition \ref{prop: S} there exists $n\ge 1$ such that, either there is some $f\in K^*$ and a valuation $\nu:K^*\to \mathbb{Z}$ such that $\nu(f)\neq 0$ and $\nu(\sigma^{nj}(f))=0$ for all but finitely many $j$, or there exists a $k$-subalgebra $S$ of $K$ having properties (i)--(iv) from the statement of Proposition \ref{prop: S}.  In the first case, we have that $K(x^n;\sigma^n)\subseteq K(x;\sigma)$ contains a free $k$-algebra generated by elements of the form $f(1-x^n)^{-1}$ and $(1-x^n)^{-1}$ (cf. \cite[Lemma 2.4 and Theorem 2.2]{BR1}), which gives the claim in this case.

Thus we assume the existence of a ring $S$ as in the statement of Proposition \ref{prop: S}.
Given a height one prime ideal $P$ of $S$, the ring $S_P$ is a discrete valuation ring since $S$ is integrally closed and noetherian and since the field of fractions of $S$ is equal to $K$, we have that this valuation gives a non-trivial discrete valuation $\nu_P$ of $K$.  There are two cases to consider.
\vskip 2mm
\emph{Case I:} There exists a height one prime ideal $P$ of $S$ such that the ideals $\{\sigma^{nj}(P)\colon j\in \mathbb{Z}\}$ are all distinct.
\vskip 2mm
In this case, we pick $f\in P$.  Since $S$ is noetherian, the intersection of infinitely many distinct height one prime ideals is $(0)$.  Consequently, the set of $j$ for which $f\in \sigma^{nj}(P)$ is finite.  In particular, the set of $j$ for which $\sigma^{nj}(f)\in P$ is finite.  It follows that $\nu_P(\sigma^{nj}(f))=0$ for all but finitely many $j$.   We now have (cf. \cite[Lemma 2.4]{BR1}) that there exists some $b\in K$ such that there are no solutions $u\in K$ to the equation $$u-\sigma^n(u)\in k+b.$$  
It now follows that $(1-x^n)^{-1}$ and $b(1-x^n)^{-1}$ generate a free $k$-subalgebra of $K(x^n;\sigma^n)\subseteq K(x;\sigma)$ \cite[Theorem 2.2]{BR1}. 
\vskip 2mm
\emph{Case II:} For every height one prime ideal $P$, there exists some $j=j(P)>0$ such that $\sigma^{nj}(P)=P$.
\vskip 2mm
Let $P$ be a height one prime that  is contained in a maximal ideal $M$ whose orbit is infinite under $\sigma^n$.  Such a maximal ideal $M$ exists by property (iv) of $S$.  Then by assumption there exists some $\ell>0$ such that $\tau:=\sigma^{n\ell}$ satisfies $\tau(P)=P$.  We then have that $S[x^{n\ell},x^{-n\ell},\tau] \cong S[t,t^{-1};\tau]$.  Since $P$ is $\tau$-invariant, the valuation $\nu_P$ extends to a valuation of $K[t;\tau]$ by declaring that
$\nu_P(a_0+\cdots +a_n t^n) = \min_{\{i\colon a_i\neq 0\}} \nu_P(a_i)$ for a nonzero element $a_0+\cdots +a_n t^n$ of $K[t;\tau]$.  Thus $\nu_P$ extends to $K(t;\tau)$.  We also have a valuation $\nu$ of $K[t;\tau]$ given by $\nu(a_it^i+\cdots +a_n t^n) =i$ when $a_i,\ldots ,a_n\in K$ and $a_i\neq 0$ (i.e., the valuation induced by the order of vanishing when $t=0$; we henceforth call this the \emph{valuation with respect} to $t$).  This then extends to a valuation of $K(t;\tau)$.   We now consider the subring $T$ of $K(t;\tau)$ consisting of all elements $f\in K(t;\tau)$ with 
$\nu_P(f)\ge 0$ and $\nu(f)\ge 0$.  Then we have a prime ideal $Q$ of $T$ consisting of all elements $f$ of $T$ for which $\nu_P(f)>0$.  Let $F=S_P/PS_P$.  Then $F$ is a finitely generated field of transcendence degree strictly less than $d$ (in fact, it is the field of fractions of the algebra $S/P$, which is finitely generated as an extension of $k$ since $S$ is a localization of a finitely generated $k$-algebra).  By assumption, $\tau$ induces an automorphism $\mu$ of $S_P/PS_P$
Then $T/Q \cong B\subseteq F(t;\mu)$.  We claim that $B$ is the subalgebra of $F(t;\mu)$ consisting of elements whose valuation with respect to $t$ is nonnegative.  To see this, observe that $K(t;\tau)$ embeds in the division ring of skew Laurent power series $K((t;\tau)$.  Then $T$ is just the set of elements of $K(t;\tau)$ whose Laurent series expansion lies in $S_P[[t;\tau]]\subseteq K((t;\tau))$.
Then the ideal $Q$ of $T$ corresponds the elements of $K(t;\tau)$ whose Laurent series expansion lies in $PS_P[[t;\tau]]$.  
Thus we have a map $T/Q\to (S_P/PS_P)[[t;\overline{\tau}]]=F[[t;\mu]]$.  Given an element $h(t)$ of $F(t;\mu)$ whose Laurent series expansion lives in $F[[t;\mu]]$, we see that we may write $h= a(t)b(t)^{-1}$ where $a(t),b(t)\in F[t;\mu]$ and $b(0)=1$.  Then by picking $\tilde{a}(t),\tilde{b}\in S_P[t;\tau]$ whose reductions mod $P$ are $a(t)$ and $b(t)$ respectively, it is straightforward to check that $\tilde{a}\tilde{b}^{-1}\in K(t;\tau)$ has a Laurent series expansion in $S_P[[t;\tau]]$ and its image mod $Q$ is equal to $h$.  Since the elements of $F(t;\mu)$ whose Laurent series expansion lives in $F[[t;\mu]]$ are exactly those elements whose valuations with respect to $t$ are nonnegative, we obtain the claim.

We now have $\mu$ is an infinite order automorphism of $F$ since the orbit of $M$ in $S$ is infinite and since $M\supseteq P$ and $P$ is $\tau$ invariant, we see $\tau^j(M)\supseteq P$ for every $j$.  Thus $\sigma^{n\ell}$ induces an infinite order automorphism of $S/P$ and hence the automorphism $\mu$ of $(S_P/P_P)=F$ is of infinite order.

Then since $F$ has transcendence degree strictly less than $d$, by our induction hypothesis, there exist elements $a,b\in B$ whose valuations with respect to $t$ are nonnegative and such that $a$ and $b$ generate a free $k$-algebra.  Then there exist $a',b'\in T$ with $\nu(a'),\nu(b')>0$ that map onto $a$ and $b$ respectively and hence $a'$ and $b'$ also generate a free $k$-subalgebra of $T$.  
This completes the proof.  
\end{proof}
We note that in the proof of Theorem \ref{thm: X}, the only way Case II can apply is if $d\ge 2$.  We now give our main result. 
\begin{cor}
\label{thm: main1}
Let $K$ be a field of characteristic zero, let $\sigma$ be an automorphism of $K$ and let $\delta$ be a $\sigma$-derivation of $K$.  Then either $D:=K(x;\sigma,\delta)$ is locally PI or $D$ contains a free algebra on two generators over its center.
\end{cor}
\begin{proof}
First, we have that $D\cong K(x;\tau)$ or $D\cong K(x;\mu)$, where $\tau$ is an automorphism of $D$ and $\mu$ is a derivation of $K$ \cite[Lemma 5.1]{BR1}.  In the case when $\mu$ is a derivation we obtain the result from \cite[Theorem 4.1]{BR1}.  In the automorphism case, if $K$ is finitely generated over $k$, we have just shown that there exist $a,b\in K(x;\tau)$ whose valuations with respect to $x$ are nonnegative that generate a free $k$-algebra unless $K(x;\tau)$ is PI, where $k=K^{\tau}$.  When $K$ is not finitely generated over $k$, then either $K$ is a direct limit of finitely generated $\tau$-invariant extensions of $k$, in which case we can reduce to the finitely generated case and we get that $K(x;\tau)$ is either locally PI or contains a free $k$-algebra on two generators; alternatively, there exists a non-finitely generated extension $K_0$ of $k$ generated by $\{\sigma^n(a)\colon n\in \mathbb{Z}\}$ for some $a\in K$, in which case we obtain the result from \cite[Proposition 3.2]{BR1}.
\end{proof}
\section{An application to division rings}
We now prove Theorem \ref{thm: main0} (see Theorem \ref{thm: main2}).  We begin with a lemma.
\begin{lem} Let $D$ be a division ring over a field $k$ and let $G\le D^*$ be a finitely generated solvable group.  If the subalgebra of $D$ generated by $G$ satisfies a polynomial identity then $G$ is abelian-by-finite. \label{lem1}
\end{lem} 
\begin{proof}
Let $R$ be the subalgebra generated by $G$.  Then $R$ is finitely generated and hence if $R$ is PI then $Q(R)$ embeds in a matrix ring over an algebraically closed field.  Since $G$ embeds in $R$ we see that $G$ embeds in a matrix ring over an algebraically closed field.  By a version of the Lie-Kolchin theorem \cite[Theorem 5.8]{Wehrfritz}, $G$ has a finite-index subgroup $H$ such that the elements of $H$ are simultaneously triangularizable.  It follows that if $h_1,h_2\in H$ then $h_1h_2-h_2h_1\in R$ is nilpotent.  Since $R$ is a domain, we see that $h_1h_2=h_2h_1$ in $R$ and so $H$ is abelian.  The result follows.
\end{proof}
  \begin{prop} Let $E$ be a locally PI division algebra over a field $k$ that is either of characteristic zero, or is uncountable and of positive characteristic.  If $E(x;\sigma)$ is not locally PI then $E(x;\sigma)$ contains a free $k$-algebra on two generators.
  \label{prop1}
  \end{prop} 
\begin{proof} We first consider the case when $E$ is PI.  In this case, $\sigma$ restricts to an automorphism of the center $Z$ of $E$.  If $Z$ has characteristic zero, then we have that $Z(x;\sigma)$ contains a free algebra on two generators unless it is locally PI by Corollary \ref{thm: main1}.  If $k$ is uncountable, then the same result holds \cite[Theorem 1.3]{BR1}.  If $Z(x;\sigma)$ contains a free algebra on two generators then so does $E(x;\sigma)$.  On the other hand, if $Z(x;\sigma)$ does not contain a free algebra, then $Z(x;\sigma)$ is locally PI.  Since $E$ is PI it is finite-dimensional over $Z$ and so $E[x,x^{-1};\sigma]$ is locally PI as it is a finite module over $Z[x,x^{-1};\sigma]$.  But then $E(x;\sigma)$ is locally PI as the quotient division ring of a locally PI ring.  Thus we obtain the result in this case. 

 So now we suppose that $E$ is locally PI and not PI and that $E(x;\sigma)$ is not locally PI.  Since $E(x;\sigma)$ is not locally PI, neither is $E[x,x^{-1};\sigma]$, and so there exists a finite subset $\{b_1,\ldots ,b_m\}\subseteq E$ such that the subalgebra of $E[x,x^{-1};\sigma]$ generated by $b_1,\ldots ,b_m, x, x^{-1}$ is not PI.   We may assume without loss of generality that $E$ is generated by 
$\{\sigma^j(b_i)\colon j\in \mathbb{Z}, i\in \{1,\ldots ,m\}\}$.  

Now let $E_{i,+}$ denote the division subalgebra of $E$ generated by $\sigma^j(b_i)$ with $j\ge 0$.  We claim that $E[x,x^{-1};\sigma]$ contain a free algebra on two generators unless each $E_{i,+}$ is finitely generated as a division algebra.  To see this, suppose that $E[x,x^{-1};\sigma]$ does not contain a free algebra and that $E_{i,+}$ is not finitely generated for some $i$.  By relabelling if necessary, we may assume that $i=1$.  Then for each $j\ge 1$ we have that $\sigma^j(b_1)\not \in k(b_1,\sigma(b_1),\ldots ,\sigma^{j-1}(b_1))$, the division subalgebra of $E$ generated over $k$ by $b_1,\ldots ,\sigma^{j-1}(b_i)$, since if this were the case by induction we would have that $E_{1,+}$ is generated by $b_1,\ldots ,\sigma^{j-1}(b_1)$ as a division algebra and hence would be finitely generated.  We now claim that the subalgebra of $E[x,x^{-1};\sigma]$ generated by $b_1x, \sigma(b_1)x^2$ is free.  To see this, observe that if $(i_1,\ldots ,i_m)\in \{1,2\}^m$ then 
$$\sigma^{i_1-1}(b_{1}) x^{i_1} \sigma^{i_2-1}(b_1) x^{i_2}\cdots \sigma^{i_m-1}(b_1) x^{i_m}$$
is equal to
$$\sigma^{i_1-1}(b_1) \sigma^{i_1+i_2-1}(b_1)\cdots \sigma^{i_1+\cdots +i_m-1}(b_1) x^{i_1+\cdots +i_m}.$$
Now define ${\rm wt}(i_1,\ldots ,i_m)$ to be $\sum i_j$ and define 
$$b(i_1,\ldots ,i_m):= \sigma^{i_1-1}(b_1) \sigma^{i_1+i_2-1}(b_1)\cdots \sigma^{i_1+\cdots +i_{m-1}-1}(b_1) $$ for $(i_1,\ldots ,i_m)\in \{1,2\}^m$. (Notice that we have omitted the final term $\sigma^{i_1+\cdots +i_m-1}(b_1)$ from the product, as it is completely determined by the weight of the string.)
Then by considering the $\mathbb{Z}$-grading on $E[x,x^{-1};\sigma]$, to show that $b_1 x$ and $\sigma(b_1) x^2$ is free it is enough to show that for a fixed natural number $n$ that
$\{ b(i_1,\ldots ,i_m)\colon m\ge 1, {\rm wt}(i_1,\ldots ,i_m)=n\}$ is linearly independent over $k$.  Without loss of generality, we may pick $n$ minimal with respect to this property.
Now suppose that we have some nontrivial linear combination that is equal to zero:
$$\sum_{m\ge 1} \sum_{{\rm wt}(i_1,\ldots ,i_m)=n} c_{i_1,\ldots ,i_m} b(i_1,\ldots ,i_m) \ = \ 0.$$
Among all strings $(i_1,\ldots ,i_m)$ with $c_{i_1,\ldots ,i_m}$ nonzero, we let $N<n$ be the maximum value of $i_1+\cdots + i_{m-1}$.  
Then  
we have
\begin{eqnarray*}
&~& \sum_{m\ge 1} \sum_{{\rm wt}(i_1,\ldots ,i_m)=n,~i_1+\cdots + i_{m-1} = N} c_{i_1,\ldots ,i_m} b(i_1,\ldots ,i_{m-1}) \sigma^{N-1}(b_1) \\ &  = & - \sum_{m\ge 1} \sum_{{\rm wt}(i_1,\ldots ,i_m)=n,~i_1+\cdots + i_{m-1} < N} c_{i_1,\ldots ,i_m} b(i_1,\ldots ,i_{m}).\end{eqnarray*}
Now
$$\sum_{m\ge 1} \sum_{{\rm wt}(i_1,\ldots ,i_m)=n,~i_1+\cdots + i_{m-1} < N} c_{i_1,\ldots ,i_m} b(i_1,\ldots ,i_{m})$$ and 
$$ \sum_{m\ge 1} \sum_{{\rm wt}(i_1,\ldots ,i_m)=n,~i_1+\cdots + i_{m-1} = N} c_{i_1,\ldots ,i_m} b(i_1,\ldots ,i_{m-1})$$
are in $k(b_1,\sigma(b_1),\ldots ,\sigma^{N-2}(b_1))$ and if the latter one is nonzero, then we have that $\sigma^{N-1}(b_1)$ is also in the division algebra generated by $b_1,\ldots ,\sigma^{N-2}(b_1)$, which we have assumed is not the case.  It follows that these must both be zero and so we have
 $$ \sum_{m\ge 1} \sum_{{\rm wt}(i_1,\ldots ,i_m)=n,~i_1+\cdots + i_{m-1} = N} c_{i_1,\ldots ,i_m} b(i_1,\ldots ,i_{m-1})=0.$$
But by our choice of $N$ we have that some $c_{i_1,\ldots ,i_m}$ appearing in this sum is nonzero and so this now gives that 
$\{ b(i_1,\ldots ,i_m)\colon m\ge 1, {\rm wt}(i_1,\ldots ,i_m)=N\}$ is linearly dependent and $N<n$, contradicting the minimality of $n$.   We thus see that we may assume that each $E_{i,+}$ is finitely generated.  A similar argument shows that we may also assume that each $E_{i,-}$ is finitely generated, where we now use backward orbits.  It follows that we may assume that $E$ is finitely generated as a division algebra and since $E$ is locally PI, we then get that $E$ is PI, which we assumed not to be the case.  This completes the proof.  
\end{proof}

We will also need the following result \cite{GS96}.
\begin{thm}\label{T:FreeAlgD} Let $D$ be a division ring with uncountable center $k$. If $D$ contains a $k$-free algebra of rank $2$, then $D$ contains the $k$-group algebra of the free group of rank $2$.
\end{thm}

We now prove our main result.
\begin{thm}
\label{thm: main2}
Let $D$ be a division ring with center $k$, and suppose that $G\le D^*$ is solvable but not locally abelian-by-finite.  Then the following statements hold:

\begin{enumerate} 
 \item if $k$ has characteristic $0$   then $D$ contains a free  $k$-algebra on two generators , and
 \item if $k$ is uncountable then $D$ contains the $k$-group algebra of the free group of rank two.
\end{enumerate}
\end{thm}
\begin{proof} Suppose that $D$ does not contain a free algebra on two generators.  Then by a result of Jategaonkar \cite[Proposition 4.13]{KL} any subalgebra of $D$ has a quotient division ring in $D$.  We may assume without loss of generality that $D$ is the quotient division ring of the subalgebra of $D$ generated by $G$ over its center.

Let $m$ be the largest index for which the $m$-th derived subgroup, $G^{(m)}$, of $G$, is not locally abelian-by-finite.  Then we may replace $G$ by $G^{(m)}$ and thus we may assume that $G'$ is locally abelian-by-finite.
Since $G'$ is locally abelian-by-finite, the division subalgebra $E$ of $D$ generated by $G'$ is locally PI.  Now every subgroup $N$ with $G'\le N\le G$ is normal.  We let $\mathcal{S}$ denote the collection of subgroups $N$ with $G'\le N\le G$ such that the division subalgebra of $D$ generated by $N$ is locally PI.  We can order $\mathcal{S}$ by inclusion and by Zorn's lemma $\mathcal{S}$ has a maximal element.  We let $H$ be a maximal element of $\mathcal{S}$.  We let $L$ denote the division subalgebra of $D$ generated by $H$.  Then $L$ is locally PI.  We note that $G\neq H$.  To see this, notice that if it were then we would have that $D$ would be locally PI.  By assumption, however, $G$ has a finitely generated subgroup $G_0$ that is not abelian-by-finite and so the division subalgebra generated by $G_0$ cannot be PI by Lemma \ref{lem1}.  Thus $D$ is not locally PI and so there is some $g\in G\setminus H$.  

Now $g$ induces an automorphism $\sigma$ of $H$ via conjugation since $H$ is normal in $G$.  Let $D_0$ denote the division subring of $D$ generated by $H$ and $g$.  By maximality of $H$, we have that $D_0$ is not locally PI.  Moreover, since $H$ is normal in $g$, we see that $gLg^{-1}=L$ and so there is a surjective homomorphism
$$L[x,x^{-1};\sigma]\to L\langle g,g^{-1}\rangle,$$ where $L\langle g,g^{-1}\rangle$ denotes the algebra generated by $L$, $g$, and $g^{-1}$ over $Z$ and $\sigma$ is the automorphism of $L$ induced by conjugation by $g$.

It follows that $D_0\cong Q(L[x,x^{-1};\sigma]/I)$ where $I$ is some ideal.  Now if $I$ is nonzero then
$L[x,x^{-1};\sigma]/I$ is a finite free $L$-module of some finite rank $d$ and thus $L[x,x^{-1};\sigma]$ embeds in $M_d(L)$, which is locally PI since $L$ is locally PI.  Since $D_0$ is not locally PI, we thus see that $I$ is zero and so
$D_0 \cong Q(L[x,x^{-1};\sigma])=L(x;\sigma)$.  Since $L(x;\sigma)$ is not locally PI and $L$ is locally PI, we see that $D_0\cong L(x;\sigma)$ contains a free algebra on two generators by Proposition \ref{prop1}.  Since $D_0\subseteq D$, we obtain the result for $D$.

Therefore, if characteristic of $k$ is $0$ we have case $(1)$, and if $k$ is uncountable, by Theorem \ref{T:FreeAlgD}, we have case $(2)$.
\end{proof}

\begin{rem}
We cannot replace locally abelian-by-finite by abelian-by-finite in the statement of Theorem \ref{thm: main2}.
\end{rem}  
For example, if we let
$E=\mathbb{C}(x_1,x_2,\ldots)$ and let $\sigma$ be the automorphism that sends $x_i\mapsto \omega_i x_i$ where $\omega_i$ is a primitive $i$-th root of unity then 
$D:=E(t;\sigma)$ is locally PI and not PI.  The reason $D$ is locally PI is that it is a direct limit of division rings $D_n=\mathbb{C}(x_1,\ldots ,x_n)(t;\sigma)$ and $\sigma$ has order dividing $n!$ on $\mathbb{C}(x_1,\ldots ,x_n)$.  On the other hand, $D$ contains $\mathbb{C}(x_n)(t;\sigma)$, whose PI degree goes to infinity as $n\to\infty$ and so $D$ is not PI.

Notice that $N:=E^*$ and $H:=\langle t\rangle$ are two subgroups of $D^*$. Since $H$ normalizes $N$ we have that $G=HN$ is a subgroup of $D^*$ and since $N\cap H = \{ 1\}$ we have a short exact sequence
$$1\to N\to G\to H\to 1.$$  Thus $G$ is metabelian.  We observe that $G$ cannot be abelian-by-finite: if it were, we would then have that the subalgebra of $D$ generated by $G$ would be PI by the Passman-Isaacs theorem \cite{Passman11}.  But the subalgebra generated by $G$ is all of $E[t,t^{-1};\sigma]$ which is locally PI and not PI.  On the other hand, $D$ cannot contain a free algebra on two generators since $E[t,t^{-1};\sigma]$ is locally PI.  We do, however, have that $G$ is locally abelian by finite by the theorem.  In this case, we can see that $G=\bigcup G_n$ where $G_n$ is the subgroup
$\mathbb{C}(x_1,\ldots ,x_n)^*H$.  Notice that conjugation by $t^{n!}$ fixes each of $x_1,\ldots ,x_n$ and hence we have a short exact sequence
$$1\to \mathbb{C}(x_1,\ldots ,x_n)^*\langle t^{n!} \rangle \to G_n \to \mathbb{Z}/n!\mathbb{Z}\to 0,$$ and
$$\mathbb{C}(x_1,\ldots ,x_n)^*\langle t^{n!}\rangle$$ is abelian and so each $G_n$ is abelian-by-finite.  Thus $G$ is a directed union of abelian-by-finite groups and so any finitely generated subgroup of $G$ is abelian-by-finite since it lies in some $G_n$.
 
\section{Results on free subgroups}
In this short section we give some results on the existence of free non-cyclic subgroups in the multiplicative group of division rings of the form $K(t;\sigma,\delta)$, where $K$ is a field and $\sigma$ and $\delta$ are respectively an automorphism and a $\sigma$-derivation of $K$.  We recall that every such ring is isomorphic to a ring of the form $K(t;\sigma)$ or $K(t;\delta)$ where $\sigma$ and $\delta$ are respectively an automorphism and a derivation of $K$.  We first consider the automorphism case.

\begin{thm} Let $K$ be a field, let $\sigma$ be an automorphism of $K$, let $k= \{ \alpha \, | \, \sigma(\alpha)=\alpha \}$ be the fixed field of $\sigma$, and let $k_{0}$ be the prime field of $K$.
Assume that $k$ has infinite transcendence degree over $k_{0}$. Then $K(t;\sigma)$ contains a free non-cyclic subgroup.
\label{thm: JG}
\end{thm}

\begin{proof}
Let $a \in K$ be such that $[a, t]=a^{-1}t^{-1}at \ne 1$. If $A= \{ a^{\sigma^{n}} \, | \, n \in \mathbb{Z} \}$ is finite, then $\overline{D}=k(a)(t)$, the division ring generated by $t$ and $a$ over $k$ is finite-dimensional over $k$. By \cite{G84}, $\overline{D}$ contains a free subgroup.

Thus we may assume that for every $a \in K \setminus k$ the set $A$ is infinite. 

Let us choose such an $a$ and let $\overline{D}$ denote the division ring $k_{0}(a)(t)$, the division ring generated by $a$ and $t$ over $k_{0}$. Let $\overline{Z}$ be the center of $\overline{D}$.

We have two possibilities:

\begin{enumerate}
 \item $k_{0}(A)$ is not finitely generated over $k_{0}$. By \cite[Prop. 2.4]{GP15a}, $\overline{D}$ contains a free subgroup.
 \item $k_{0}(A)$ is finitely generated over $k_{0}$. Then the transcendence degree of $\overline{Z}$ over $k_{0}$ is finite, and as in the proof of \cite[Theorem 5.1]{GP15a}, we can find two elements $X$ and $Y$ in $k$, algebraically independent over $\overline{D}$. By \cite[Proposition 5.8]{GP15a}, $\overline{D}(X,Y)$ contains a free subgroup. And so $D$ also does.
\end{enumerate}

\end{proof}
We remark that the derivation case is considerably easier and in this situation one has that the multiplicative group of $K(t;\delta)$ contains a free non-cyclic subgroup whenever $\delta$ is a nonzero derivation of $K$.  
Since there is some $a\in K$ with $\delta(a)\neq 0$, we can set $u=t\delta(a)^{-1}$.  Then since $a$ and $\delta(a)$ commute, we have
$$[u,a] = [t,a]\delta(a)^{-1} = 1.$$
Thus if $K$ has characteristic $p>0$ we have that the division ring generated by $u$ and $a$ is not commutative and finite-dimensional over its center, and thus its multiplicative group contains a free non-cyclic group by \cite{G84}.  If $K$ has characteristic zero, then the division ring generated by $u$ and $a$ is isomorphic to the Weyl algebra and hence its multiplicative group contains a free non-cyclic group [ \cite{FG15} Theorem 1.3]
Thus for the question of existence of free non-cyclic subgroups in the multiplicative group of $D=K(t;\sigma,\delta)$, we know that it is always true, unless $D$ is a field (in which case it cannot have one), or when $\sigma$ has infinite order and the fixed field of $\sigma$ has finite transcendence degree over the prime field.


\begin{thebibliography}{50}



 

  




\bibitem{Amerik} E. Amerik,
\emph{Existence of non-preperiodic algebraic points for a rational self-map of infinite order.} 
\emph{Math. Res. Lett.} {\bf 18} (2011), no. 2, 251--256.

\bibitem{BGT} J. P. Bell, D. Ghioca, and T. J. Tucker, 
\emph{The dynamical Mordell-Lang problem for \'etale maps.} 
\emph{Amer. J. Math.} {\bf 132} (2010), no. 6, 1655--1675. 

\bibitem{BR1} J. P. Bell and D. Rogalski, \emph{Free subalgebras of quotient rings of Ore extensions.} Algebra Number Theory {\bf 6} (2012), no. 7, 1349--1367.

\bibitem{BR2} \bysame, \emph{Free subalgebras of division algebras over uncountable fields.} Math. Z. {\bf 277} (2014), no. 1--2, 591--609.

\bibitem{Chiba96}
K. Chiba, \emph{Free subgroups and free subsemigroups of division rings.} J. Algebra 
 \textbf{184} (1996), 570--574.
 
 \bibitem{Eisenbud} D. Eisenbud, \emph{Commutative algebra. With a view toward algebraic geometry,}
Graduate Texts in Mathematics 150 Springer-Verlag, New York, 1995.

\bibitem{FGM05}
V. O. Ferreira, J. Z. Gon\c calves, and A. Mandel, \emph{Free symmetric and unitary pairs in division rings with involution.} Internat. J. Algebra Comp.  \textbf{15}, 1 (2005), 15--36.

\bibitem{FG15}
 V. O. Ferreira and J. Z. Gon\c calves. \emph{Free symmetric and unitary pairs in division rings infinite-dimensional over their enters.} Israel J. Math. (to appear)
 
\bibitem{FGS} L. M. Figuerido, J. Z. Gon\c calves, M. Shirvani, \emph{Free group algebras in certain division rings.} J. Algebra {\bf 185} (2) (1996) 298--311.

\bibitem{G84}
 J. Z. Gon\c calves \emph{Free groups in subnormal subgroups and the residual  nilpotence of the group of units of group rings}, Canad. Math. Bull.  \textbf{27} (1984), 365--370.
 

 \bibitem{GL15}
 J. Z. Gon\c calves and A. I. Lichtman, \emph{Free subgroups in division rings generated by group rings of soluble groups.} Int. J. Algebra Comp. \textbf{24}, 8 (2014), 1127--1140.


\bibitem{GP15a}
 J. Z. Goncalves and D. S. Passman \emph{Free groups in normal subgroups of the multiplicative group of a division ring.} J.  Algebra 440 (2015), 128-144.

\bibitem{GP14}
\bysame  \emph{Explicit free groups in division rings.} Proc. Amer. Math. Soc. \textbf{143}, 2 (2015), 459--468.

\bibitem{GS96}
 J. Z. Gon\c calves and M. Shirvani, \emph{On free group algebras in division rings with uncountable center.} Proc. Amer. Math. Soc. \textbf{124}, 3 (1996), 685-687. 

\bibitem{GS12} 
\bysame , \emph{A survey on free objects in division rings and in division rings with an involution.} Comm. Algebra \textbf{40} (2012), 1704--1723.


\bibitem{H}  E. Hrushovski, \emph{The elementary theory of the Frobenius automorphisms}, preprint, arXiv:math/0406514v1.


\bibitem{KL} G. R. Krause, T. H. Lenagan, \emph{Growth of algebras and Gelfand-Kirillov dimension.} Revised edition. Graduate Studies in Mathematics, 22. American Mathematical Society, Providence, RI, 2000.

\bibitem{Lichtman78}
 A. I. Lichtman, \emph{Free subgroups of normal subgroups of the multiplicative group of skew fields.} Proc. Amer. Math. Soc.  \textbf{7}, 2 (1978), 174--178.
  
\bibitem{Licht} \bysame, \emph{Free subalgebras in division rings generated by universal enveloping algebras.} Algebra Colloq. {\bf 6} (2) (1999) 145--153.


\bibitem{L} M. Lorenz, 
\emph{Group rings and division rings.} Methods in ring theory, Antwerp (1983), 265--280, 
NATO Adv. Sci. Inst. Ser. C Math. Phys. Sci., 129, Reidel, Dordrecht, 1984.

\bibitem{Lor} \bysame, \emph{On free subalgebras of certain division algebras.} Proc. Amer. Math. Soc. {\bf 98} (3) (1986) 401--405.

\bibitem{ML} L. Makar-Limanov, \emph{The skew field of fractions of the Weyl algebra contains a free noncommutative subalgebra.} Comm. Algebra {\bf 11} (1983), no. 17, 2003--2006.

\bibitem{ML15} \bysame, \emph{On free subsemigroups of skew fields.} Proc. Amer. Math. Soc. {\bf 91} (2) (1984) 189--191.

\bibitem{ML2} \bysame, \emph{On group rings of nilpotent groups.} Israel J. Math. {\bf 48} (1984),  no. 2-3, 244--248.

\bibitem{ML3} L. Makar-Limanov and P. Malcolmson, \emph{Free subalgebras of enveloping fields.} Proc. Amer. Math. Soc. {\bf 111} (1991), no. 2, 315--322.

\bibitem{Passman11} D. S. Passman, \emph{The algebraic structure of group rings.} Dover, Mineola, New York, 2011.


\bibitem{RV} Z. Reichstein and N. Vonessen, \emph{Free subgroups in division algebras.} Comm. Algebra {\bf 23} (6) (1995) 2181--2185.

\bibitem{SG} M. Shirvani and J. Z. Gon\c calves, 
\emph{Large free algebras in the ring of fractions of skew polynomial rings.}  
J. London Math. Soc. (2) {\bf 60} (1999), no. 2, 481--489.

\bibitem{Wehrfritz} B. A. F. Wehrfritz, \emph{Infinite linear groups.} Ergebnisse der Matematik und ihrer Grenzgebiete, Band 76. Springer-Verlag, New York-Heidelberg, 1973.

\end{thebibliography}
\end{document}